\documentclass[11pt]{article}

\usepackage[a4paper,margin=28mm]{geometry}
\usepackage{amsmath,amssymb,amsthm,mathtools}
\usepackage{booktabs}
\usepackage{array}
\usepackage{microtype}
\usepackage[hidelinks]{hyperref}
\usepackage[nameinlink,noabbrev]{cleveref}

\newtheorem{theorem}{Theorem}[section]
\newtheorem{lemma}[theorem]{Lemma}
\newtheorem{proposition}[theorem]{Proposition}
\newtheorem{corollary}[theorem]{Corollary}

\theoremstyle{definition}

\newtheorem{remark}[theorem]{Remark}

\title{Chromatic Extremal Thresholds and the Multipartite $K_4$-Free Problem}
\author{Yuuki Kasugai\\{\small Japan}}
\date{}

\begin{document}
\maketitle

\begin{abstract}
For positive integers $n,r,t$, let $\delta(n,r,t)$ denote the maximum possible minimum degree of a balanced $r$-partite graph with parts of size $n$ and chromatic number at most $t$. Lo, Treglown and Zhao~\cite{LTZ} established a general upper bound for this parameter and used it, together with explicit constructions, to determine the corresponding multipartite clique threshold up to an additive constant in a broad parameter range. I determine the chromatic parameter throughout the range
\[
r=mt-a,\qquad
m\ge2,\quad t\ge3,\quad
2\le a\le \min\{m,t-1\}.
\]
The answer differs from the Lo--Treglown--Zhao upper bound by at most one. I give an explicit arithmetic criterion deciding when this one-unit correction occurs. The proof reduces the problem to an integer matrix extremum. In the boundary case, equality forces the supports of all mixed rows to form a spanning star, after which the only remaining obstruction is a divisibility condition. Combining this formula with the Andrásfai--Erdős--Sós theorem~\cite{AES} sharpens the known equality range for
$f(n,r,t+1)=\delta(n,r,t)$. In particular, for $t=3$ it removes the remaining size restrictions at $r=10$ and $r=13$. Together with the $r=7$ result in~\cite{Kasugai}, the classical $r=4$ case, and the known congruence classes, this gives a formula for the multipartite $K_4$-free problem for every admissible $r\ge4$ and every $n\ge1$.
\end{abstract}

\section{Introduction}

Let $f(n,r,t+1)$ denote the largest possible minimum degree of a balanced
$r$-partite graph with parts of size $n$ containing no copy of $K_{t+1}$.
The multipartite clique problem was posed by Bollob\'as, Erd\H{o}s and
Szemer\'edi in 1975~\cite{BES}.
Lo, Treglown and Zhao~\cite{LTZ} introduced the related parameter
\[
\delta(n,r,t)
=
\max\{\delta(G):G\text{ is balanced $r$-partite with parts of size $n$ and }
\chi(G)\le t\}.
\]
For $(m-1)t<r<mt$, they proved~\cite[Proposition~5.1]{LTZ}
\begin{equation}\label{eq:LTZ-upper}
\delta(n,r,t)
\le
(r-1)n-
\left\lceil
\frac{(m-1)(r-1)n}{mt-2}
\right\rceil .
\end{equation}
They also constructed graphs giving a lower bound whose integer rounding may differ from
\eqref{eq:LTZ-upper} by an additive constant.

The case $r=t+1$ had been settled by Haxell and Szab\'o~\cite{HS}
and Szab\'o and Tardos~\cite{ST}.

In earlier work, I determined the remaining case $f(n,7,4)$~\cite{Kasugai}.
A central step in that argument was the evaluation of $\delta(n,7,3)$ through an integer
matrix formulation.  The present paper studies the integer extremal structure behind that
calculation in a general parameter range; the $r=7$, $t=3$ matrix extremum appears as a
special case of the formula below.

This paper has two main results.  First, for
\[
r=mt-a,\qquad
m\ge2,\quad t\ge3,\quad 2\le a\le\min\{m,t-1\},
\]
I determine $\delta(n,r,t)$ for every $n$.  If
\[
d=\left\lceil\frac{(m-1)(r-1)n}{mt-2}\right\rceil,
\qquad d=(m-1)s+T,
\quad s,T\in\mathbb{Z},\quad 0\le T\le m-2,
\]
and
\[
H=n+td+(t-2)s-rn,
\]
then the correction to the Lo--Treglown--Zhao bound is always either zero or one, and is
specified by the sign of $H$ and the divisibility condition $(n-s)\mid T$.

Second, the specialization $t=3$ determines the corresponding multipartite $K_4$-free
threshold.  The $r=4$ case and the congruence classes $r\equiv0,-1\pmod3$ were already known
\cite{LTZ}, while the case $r=7$ was established in~\cite{Kasugai}.  The value of
$\delta(n,r,3)$ obtained here removes the remaining size restrictions at $r=10$ and $r=13$
and, together with the known equality $f(n,r,4)=\delta(n,r,3)$ for $r\ge16$ with
$r\equiv1\pmod3$~\cite{LTZ}, yields the value of $f(n,r,4)$ for every $r\ge4$ and
every $n\ge1$.  Thus the $t=3$ case of the multipartite clique problem posed in 1975 is
determined.

\section{The integer matrix formulation}

Write $[r]:=\{1,\dots,r\}$.  A row of a matrix will be called \emph{mixed} if it has at
least two positive entries.

For an $r\times t$ nonnegative integer matrix
\[
A=(a_{ij}),
\qquad
\sum_{j=1}^t a_{ij}=n
\quad (i\in[r]),
\]
write
\[
c_j:=\sum_{i=1}^r a_{ij}
\]
for its $j$th column sum and define
\[
D(A):=
\max_{\substack{i\in[r],\,j\in[t]\\ a_{ij}>0}}
(c_j-a_{ij}).
\]

\begin{lemma}[Matrix formulation]\label{lem:matrix}
For all $n,r,t$,
\[
\delta(n,r,t)
=
(r-1)n-\min_A D(A),
\]
where the minimum is over all nonnegative integer $r\times t$ matrices whose row sums are $n$.
\end{lemma}

\begin{proof}
Let $G$ be a balanced $r$-partite graph with vertex classes
$V_1,\dots,V_r$ and a proper $t$-colouring with colour classes
$W_1,\dots,W_t$. Put
\[
a_{ij}:=|V_i\cap W_j|.
\]
Adding every edge whose endpoints lie in different vertex classes and in different colour
classes preserves both $r$-partiteness and $t$-colourability, and cannot decrease the minimum
degree. Thus it is enough to consider the completed graph determined by $A$.

If $x\in V_i\cap W_j$, then
\[
d(x)
=
rn-|V_i|-|W_j|+|V_i\cap W_j|
=
(r-1)n-c_j+a_{ij}.
\]
Hence the minimum degree of the completed graph is
\[
(r-1)n-D(A).
\]
Conversely, every admissible matrix $A$ can be realised by partitioning each $V_i$ into cells
of sizes $a_{ij}$ and joining exactly the pairs lying in distinct vertex classes and distinct
colour classes. This proves the equality.
\end{proof}

\section{Evaluation}

Throughout this section assume
\[
m\ge2,\qquad t\ge3,\qquad
2\le a\le\min\{m,t-1\},
\qquad
r=mt-a.
\]
Define
\begin{equation}\label{eq:ddef}
d:=
\left\lceil
\frac{(m-1)(r-1)n}{mt-2}
\right\rceil ,
\end{equation}
and write
\begin{equation}\label{eq:division}
d=(m-1)s+T,
\qquad
s=\left\lfloor\frac{d}{m-1}\right\rfloor,
\qquad
0\le T\le m-2.
\end{equation}
Finally define
\begin{equation}\label{eq:Hdef}
H:=n+t d+(t-2)s-rn.
\end{equation}

\begin{theorem}\label{thm:main}
With the notation above,
\[
\delta(n,r,t)
=
(r-1)n-d-\varepsilon,
\]
where $\varepsilon\in\{0,1\}$ is given by
\[
\varepsilon=
\begin{cases}
1,& H<0,\\
1,& H=0\text{ and }(n-s)\nmid T,\\
0,& \text{otherwise}.
\end{cases}
\]
Equivalently,
\[
\min_A D(A)=d+\varepsilon.
\]
\end{theorem}

The proof is divided into several lemmas.

\begin{lemma}\label{lem:basic}
One has
\[
1\le s\le n.
\]
Moreover, if $H\le0$, then
\[
d<(m-1)n
\qquad\text{and hence}\qquad
s\le n-1.
\]
\end{lemma}

\begin{proof}
Since $a\ge2$,
\[
\frac{(m-1)(r-1)n}{mt-2}
<
(m-1)n,
\]
so $d\le(m-1)n$ and therefore $s\le n$.

For the lower bound, it is enough to take $n=1$. In this case,
\[
(m-1)(r-1)-(m-2)(mt-2)
=
m(t-a+1)+a-3>0,
\]
because $a\le t-1$. Thus the quantity inside the ceiling in
\eqref{eq:ddef} is greater than $m-2$, so $d\ge m-1$ and $s\ge1$.

If $d=(m-1)n$, then $s=n$ and $T=0$, so
\[
H
=
n+t(m-1)n+(t-2)n-rn
=
(a-1)n>0.
\]
Therefore $H\le0$ implies $d<(m-1)n$.
\end{proof}

For an integer $D\ge0$, define
\[
s_D:=\left\lfloor\frac{D}{m-1}\right\rfloor,
\qquad
B_D:=D+s_D
=
\left\lfloor\frac{mD}{m-1}\right\rfloor,
\]
and
\[
H_D:=n+tD+(t-2)s_D-rn.
\]
Thus $s_d=s$ and $H_d=H$.

\begin{lemma}[Column bound]\label{lem:column}
If $A$ is admissible and $D(A)\le D$, where $D\ge d$, then every column sum satisfies
\[
c_j\le B_D.
\]
\end{lemma}

\begin{proof}
First note that
\[
\frac{m}{m-1}
\cdot
\frac{(m-1)(r-1)n}{mt-2}
=
\frac{m(r-1)n}{mt-2}
>
(m-1)n,
\]
because
\[
m(r-1)-(m-1)(mt-2)
=
m(t-a+1)-2>0.
\]
Since $D\ge d$, this implies
\[
B_D=\left\lfloor\frac{mD}{m-1}\right\rfloor
\ge(m-1)n.
\]

Fix a column $j$ and let $k$ be the number of its positive entries.
If $k\le m-1$, then $c_j\le(m-1)n\le B_D$.

If $k\ge m$, the smallest positive entry in column $j$ is at most $c_j/k$.
Hence
\[
D
\ge
c_j-\min_{i:a_{ij}>0}a_{ij}
\ge
\frac{k-1}{k}c_j
\ge
\frac{m-1}{m}c_j.
\]
Thus
\[
c_j\le \left\lfloor\frac{mD}{m-1}\right\rfloor=B_D.
\]
\end{proof}

\begin{lemma}[Star construction]\label{lem:construction}
Let $D\ge d$, write
\[
D=(m-1)s_D+T_D,\qquad 0\le T_D\le m-2.
\]
If $H_D>0$, then there exists an admissible matrix $A$ with $D(A)\le D$.

If $H_D=0$ and
\[
(n-s_D)\mid T_D,
\]
then there also exists an admissible matrix $A$ with $D(A)\le D$.
\end{lemma}

\begin{proof}
First consider the case $D\ge(m-1)n$.  Partition the $r$ rows into $t$ groups, with $t-a$ groups of size $m$ and $a$ groups of size $m-1$, and assign a distinct colour to each group.  The number of rows used is
\[
(t-a)m+a(m-1)=mt-a=r.
\]
For this matrix every nonzero entry equals $n$, and a colour supported on $k\le m$
rows contributes deficit $(k-1)n\le(m-1)n\le D$.  Hence $D(A)\le D$.

Hence, throughout the remainder of the proof, assume that
\[
D<(m-1)n.
\]
In particular $s_D\le n-1$, so $n-s_D\ge1$; moreover, if $T_D>0$, then
$s_D+1\le n$.

Use colour $t$ as a centre and colours $1,\dots,t-1$ as leaves.
Allocate $m$ rows to each leaf; the remaining
\[
r-m(t-1)=m-a
\]
rows will be centre-only rows.

First suppose $H_D>0$. For each leaf, place $s_D+1$ vertices in that leaf in
$T_D$ of its $m$ rows, and $s_D$ vertices in the leaf in the other
$m-T_D$ rows. Put the remaining vertices of each such row in the centre.
The $m-a$ remaining rows are entirely in the centre.

The column sum of each leaf is
\[
T_D(s_D+1)+(m-T_D)s_D
=
ms_D+T_D
=
D+s_D
=
B_D.
\]
Thus the deficit at the leaf entry of a mixed row is at most $D$.

The centre column sum is
\[
c_t=rn-(t-1)B_D.
\]
For a row in which the leaf entry equals $s_D+1$, the centre deficit is
\[
c_t-(n-s_D-1).
\]
The inequality that this is at most $D$ is precisely
\[
H_D\ge1.
\]
Since $H_D$ is an integer and $H_D>0$, this holds.
For a row with leaf entry $s_D$, the centre deficit satisfies
\[
c_t-(n-s_D)
=\bigl[c_t-(n-s_D-1)\bigr]-1
\le D-1.
\]
For a centre-only row,
\[
c_t-n
=\bigl[c_t-(n-s_D)\bigr]-s_D
\le D-1-s_D
\le D.
\]
Hence $D(A)\le D$.

Now suppose $H_D=0$ and put
\[
h_D:=n-s_D.
\]
By assumption, $u:=T_D/h_D$ is an integer.
For each leaf, take $u$ leaf-only rows and $m-u$ mixed rows.
In each mixed row put $s_D$ vertices in the leaf and $h_D$ vertices in the centre.
Again use $m-a$ centre-only rows.

Each leaf column has sum
\[
un+(m-u)s_D
=
ms_D+u(n-s_D)
=
ms_D+T_D
=
B_D.
\]
Since $H_D=0$,
\[
rn=n+2D+(t-2)B_D,
\]
and therefore
\[
c_t
=
rn-(t-1)B_D
=
n+D-s_D.
\]
The centre deficit in every mixed row is exactly
\[
c_t-(n-s_D)=D.
\]
The other deficits are at most $D$, so $D(A)\le D$.
\end{proof}

It remains to prove that the boundary divisibility condition is necessary.

\begin{lemma}[Boundary structure]\label{lem:boundary}
Assume $H=0$ and let $A$ be admissible with $D(A)\le d$.
Then the supports of all rows containing at least two colours form a spanning star on the
$t$ colours.

For every leaf colour $j$, every mixed row containing $j$ has leaf entry exactly $s$.
Consequently
\[
(n-s)\mid T.
\]
\end{lemma}

\begin{proof}
By Lemma~\ref{lem:basic}, $d<(m-1)n$.
Thus not all rows can be monochromatic: otherwise, since
$r=mt-a>(m-1)t$, some colour would contain at least $m$ entire rows, giving
$D(A)\ge(m-1)n>d$.

Let a row use exactly $k\ge2$ colours.
For each colour present in the row,
\[
c_j\le d+a_{ij},
\]
while every absent column is at most
\[
B:=B_d=d+s
\]
by Lemma~\ref{lem:column}. Summing gives
\[
rn
\le
n+kd+(t-k)B.
\]
Since $H=0$,
\[
rn=n+2d+(t-2)B.
\]
As $B-d=s>0$, the last two displays force $k=2$, and equality holds throughout.
Thus, if a mixed row has support $\{u,v\}$, then
\begin{equation}\label{eq:equality-columns}
c_w=B
\quad\text{for every }w\notin\{u,v\},
\end{equation}
and
\[
c_u=d+a_{iu},\qquad
c_v=d+a_{iv}.
\]

Let $\Gamma$ be the graph on the $t$ colours whose edges are the supports of the mixed rows.
First, not all columns can have sum $B$.
If $c_j=B$ for every $j$, then $tB=rn$.
Using $H=0$ and $B=d+s$, this yields $n=2s$.
Writing $d=(m-1)s+T$ and substituting into $H=0$ gives
\[
(mt-2a)s=tT.
\]
But $a\le t-1$, so
\[
\frac{mt-2a}{t}>m-2.
\]
Since $s\ge1$, this forces $T>m-2$, contradicting $T\le m-2$.

Since not all columns have sum $B$, choose a colour $z$ with $c_z<B$.
By \eqref{eq:equality-columns}, every mixed edge must contain $z$: if a mixed edge avoided
$z$, then that equation would force $c_z=B$.  Hence all mixed edges have the common centre
$z$.

Next, this star is spanning. First observe that
\[
B\ge(m-1)n.
\]
Equality is impossible when $H=0$: if $B=(m-1)n$, then
$d=B-s=(m-1)n-s$, and substituting into $H=0$ gives
\[
(a+1-t)n=2s,
\]
which is impossible because $a\le t-1$ and $s>0$.
Thus
\[
(m-1)n<B.
\]
Also $d<(m-1)n$ implies
\[
B=d+s<mn.
\]
Hence $B$ is not a multiple of $n$.

If a colour were not incident with any edge of $\Gamma$, it would appear only in
monochromatic rows. But by \eqref{eq:equality-columns}, any such colour has column sum $B$,
which would then be a multiple of $n$, a contradiction.
Thus $\Gamma$ is a spanning star.

Let colour $0$ be its centre and let $j$ be a leaf.
Since $t\ge3$, there is another leaf, so a mixed edge avoiding $j$ shows that $c_j=B$.
For a mixed row containing leaf $j$, equality in
$c_j\le d+a_{ij}$ gives
\[
a_{ij}=B-d=s.
\]
Let $u$ be the number of monochromatic rows of colour $j$, and let $K$ be the total number
of rows in which colour $j$ appears. Then
\[
B=(K-u)s+un=Ks+u(n-s).
\]
Since $B>(m-1)n$, it follows that $K\ge m$. On the other hand,
\[
B=ms+T,
\]
so
\begin{equation}\label{eq:leaf-eqn}
(K-m)s+u(n-s)=T.
\end{equation}

If $(n-s)\nmid T$, then $n-s\ge2$.
In this case, $s>T$.
Suppose instead that $s\le T$.
From $H=0$ one obtains
\[
(mt-2)(n-s)=(a-1)n+tT.
\]
Since $n=s+(n-s)\le T+(n-s)$ and $T\le m-2$,
\[
(mt-a-1)(n-s)
\le
(a+t-1)(m-2).
\]
As $n-s\ge2$, this implies
\[
2(mt-a-1)
\le
(a+t-1)(m-2).
\]
But
\[
2(mt-a-1)-(a+t-1)(m-2)
=
m(t-a+1)+2t-4>0,
\]
a contradiction. Thus $s>T$.

Equation \eqref{eq:leaf-eqn} and $K\ge m$ now force $K=m$, and hence
\[
u(n-s)=T.
\]
Therefore $(n-s)\mid T$, proving the final assertion.
\end{proof}

\begin{lemma}[One-step sufficiency]\label{lem:dplus1}
If $D=d$ is not attainable, then $D=d+1$ is attainable.
\end{lemma}

\begin{proof}
Put
\[
\Delta^*
=
\frac{(m-1)(r-1)n}{mt-2}.
\]
The continuous identity
\[
n+t\Delta^*+(t-2)\frac{\Delta^*}{m-1}-rn=0
\]
holds directly from the definition of $\Delta^*$.
Since $d=\lceil\Delta^*\rceil$ and
\[
\left\lfloor\frac{d}{m-1}\right\rfloor
>
\frac{d}{m-1}-1,
\]
it follows that
\[
H_d>-(t-2).
\]
As $H_d$ is an integer,
\[
H_d\ge-(t-3).
\]

For every integer $D$,
\[
H_{D+1}-H_D
=
t+(t-2)
\left(
\left\lfloor\frac{D+1}{m-1}\right\rfloor
-
\left\lfloor\frac{D}{m-1}\right\rfloor
\right)
\ge t.
\]
Therefore
\[
H_{d+1}\ge3.
\]
If $d$ is not attainable, then $H_d\le0$ by Lemma~\ref{lem:construction}.
In particular $d<(m-1)n$, so $d+1\le(m-1)n$ and the star construction is valid.
Lemma~\ref{lem:construction} with $D=d+1$ gives an admissible matrix satisfying
$D(A)\le d+1$.
\end{proof}

\begin{proof}[Proof of Theorem~\ref{thm:main}]
The upper bound of Lo--Treglown--Zhao~\cite[Proposition~5.1]{LTZ} for $\delta(n,r,t)$ is equivalent to the following
lower bound on the matrix deficit:
\[
D(A)\ge d
\]
for every admissible matrix $A$.

If $H>0$, Lemma~\ref{lem:construction} constructs a matrix with $D(A)\le d$.
If $H=0$ and $(n-s)\mid T$, the second construction in
Lemma~\ref{lem:construction} does the same.

If $H<0$, the argument used in Lemma~\ref{lem:boundary} for a mixed row gives
\[
rn\le n+2d+(t-2)(d+s)<rn,
\]
unless all rows are monochromatic; but the latter possibility gives
$D(A)\ge(m-1)n>d$. Hence $d$ is impossible.

If $H=0$ and $(n-s)\nmid T$, Lemma~\ref{lem:boundary} shows that $d$ is impossible.

Finally, Lemma~\ref{lem:dplus1} shows that whenever $d$ is impossible, $d+1$ is attainable.
Thus the minimum matrix deficit is exactly $d+\varepsilon$, as claimed.
Lemma~\ref{lem:matrix} converts this to the stated formula for $\delta(n,r,t)$.
\end{proof}

\section{Consequences for multipartite clique thresholds}

The formula for $\delta$ also gives a convenient integer form of the
Andr\'asfai--Erd\H{o}s--S\'os connection~\cite{AES} used by
Lo--Treglown--Zhao~\cite{LTZ}.

\begin{corollary}[Integer AES criterion]\label{cor:aes-integer}
Under the hypotheses of Theorem~\ref{thm:main}, if
\[
(r-1)n-d-\varepsilon
\ge
\left\lfloor\frac{3t-4}{3t-1}rn\right\rfloor,
\]
then
\[
f(n,r,t+1)=(r-1)n-d-\varepsilon.
\]
\end{corollary}

\begin{proof}
By Theorem~\ref{thm:main}, the left-hand side is $\delta(n,r,t)$.  Suppose for a
contradiction that $f(n,r,t+1)>\delta(n,r,t)$.  Since both quantities are integers,
\[
f(n,r,t+1)\ge\delta(n,r,t)+1
>\frac{3t-4}{3t-1}rn.
\]
The Andr\'asfai--Erd\H{o}s--S\'os theorem, in the form used in
\cite[Corollary~3.2]{LTZ}, then forces an extremal $K_{t+1}$-free graph to be
$t$-colourable, contradicting the definition of $\delta(n,r,t)$.
\end{proof}

The following is immediate from Theorem~\ref{thm:main} and the further equality criteria of
Lo--Treglown--Zhao~\cite[Corollaries~3.3 and~4.3]{LTZ}.

\begin{corollary}\label{cor:f}
Under the hypotheses of Theorem~\ref{thm:main}, suppose in addition that either
\[
r\ge a(3t-1),
\]
or
\[
\frac{r}{t(3t-1)(m-1)}
-
\frac{a}{t(m-1)}
+
\frac{a-1}{tm-2}
\ge
\frac1n.
\]
Then
\[
f(n,r,t+1)
=
(r-1)n-d-\varepsilon,
\]
with $d$ and $\varepsilon$ as in Theorem~\ref{thm:main}.
\end{corollary}

\begin{proof}
Under either displayed condition, Lo--Treglown--Zhao~\cite[Corollaries~3.3 and~4.3]{LTZ} proved
\[
f(n,r,t+1)=\delta(n,r,t).
\]
Apply Theorem~\ref{thm:main}.
\end{proof}

\begin{remark}
For $t=3$ and $a=2$, the theorem treats
\[
r=3m-2=4,7,10,13,\dots
\]
within the valid range $r>t$. In particular, the $r=7$ instance recovers the integer
three-colourable extremum obtained in~\cite{Kasugai}.
\end{remark}

\section{The \texorpdfstring{$K_4$}{K4}-free case}

I now specialize to $t=3$.  Combined with the previously solved $r=t+1$ case~\cite{LTZ} and the
$r=7$ result in~\cite{Kasugai}, the chromatic extremum obtained above determines the $K_4$-free
instance of the 1975 multipartite clique problem.  The only additional work needed is to remove
the size restrictions at $r=10$ and $r=13$; these are handled for every $n$ using the
value of $\delta(n,r,3)$.

Let
\[
r=3m-2,\qquad m\ge3,
\]
and write
\[
n=qr+\rho,\qquad 0\le \rho<r.
\]
Define
\[
\eta(n,r)=1
\]
if there is an integer $h$ with
\[
1\le h\le m-2,\qquad
\rho=3h+1,\qquad
q+1\nmid m-h-1,
\]
and define $\eta(n,r)=0$ otherwise.

\begin{corollary}\label{cor:t3-delta}
For $r=3m-2$ with $m\ge3$,
\[
\delta(n,r,3)
=
(r-1)n-
\left\lceil
\frac{(m-1)(r-1)n}{r}
\right\rceil
-\eta(n,r).
\]
\end{corollary}

\begin{proof}
Apply Theorem~\ref{thm:main} with $t=3$ and $a=2$, so $r=3m-2$ and
\[
d=\left\lceil\frac{(m-1)(r-1)n}{r}\right\rceil.
\]
Write $n=qr+\rho$ with $0\le\rho<r$, and put
\[
k:=\left\lfloor\frac{(m-1)\rho}{r}\right\rfloor.
\]
Since $\gcd(m-1,r)=1$, the number $(m-1)\rho/r$ is nonintegral whenever
$0<\rho<r$.  Hence
\[
d=(m-1)(n-q)-k.
\]
If $k=0$, then $s=n-q$, $T=0$, and
\[
H=n+3d+s-rn=\rho\ge0.
\]
In particular the case $\rho=0$ has $H=0$ but $T=0$, so it produces no
one-unit correction.

Now suppose $k\ge1$.  Since $0<k\le m-2$,
\[
s=n-q-1,\qquad T=m-k-1,
\]
and therefore
\[
H=\rho-3k-1.
\]
Moreover, because $k<(m-1)\rho/r$ and $r=3m-2$,
\[
\rho>\frac{kr}{m-1}=3k+\frac{k}{m-1},
\]
so the integrality of $\rho$ gives $\rho\ge3k+1$.  Thus $H<0$ never occurs.
Furthermore,
\[
H=0
\iff
\rho=3k+1.
\]
Writing $h=k$, this is exactly
\[
\rho=3h+1,\qquad 1\le h\le m-2.
\]
In this boundary case
\[
n-s=q+1,\qquad T=m-h-1.
\]
Theorem~\ref{thm:main} now says that the correction is one precisely when
\[
q+1\nmid m-h-1,
\]
which is the definition of $\eta(n,r)$.
\end{proof}

\begin{proposition}\label{prop:r10r13}
For every $n\ge1$,
\[
f(n,10,4)=\delta(n,10,3)
\qquad\text{and}\qquad
f(n,13,4)=\delta(n,13,3).
\]
\end{proposition}

\begin{proof}
By Corollary~\ref{cor:aes-integer}, it is enough to prove
\[
\delta(n,r,3)\ge
\left\lfloor\frac58rn\right\rfloor.
\]

First let $r=10$, so $m=4$.  Put
\[
d=\left\lceil\frac{27n}{10}\right\rceil.
\]
By Corollary~\ref{cor:t3-delta},
\[
\delta(n,10,3)=9n-d-\eta(n,10).
\]
Thus it is enough to prove
\[
d+\eta(n,10)\le\left\lceil\frac{11n}{4}\right\rceil.
\]
Write $n=10q+\rho$ with $0\le\rho\le9$.  Then
\[
d=27q+\left\lceil\frac{27\rho}{10}\right\rceil
\]
and
\[
\left\lceil\frac{11n}{4}\right\rceil
=
27q+
\left\lceil
\frac q2+\frac{11\rho}{4}
\right\rceil.
\]
Since $11\rho/4\ge27\rho/10$, the desired inequality is immediate except possibly when
$\eta(n,10)=1$.  The criterion in Corollary~\ref{cor:t3-delta} shows that this occurs
only in the following two cases:
\[
\rho=4,\ q\ge2,
\qquad\text{or}\qquad
\rho=7,\ q\ge1.
\]
For $\rho=4$, the two ceilings at $q=0$ are both $11$, while $q\ge2$ contributes at least
one additional unit through the term $q/2$.  For $\rho=7$,
\[
\left\lceil\frac{77}{4}\right\rceil
-
\left\lceil\frac{189}{10}\right\rceil
=
20-19=1.
\]
Hence the required inequality holds for all $n$.

Now let $r=13$, so $m=5$, and put
\[
d=\left\lceil\frac{48n}{13}\right\rceil.
\]
It is enough to prove
\[
d+\eta(n,13)\le\left\lceil\frac{31n}{8}\right\rceil.
\]
Write $n=13q+\rho$, with $0\le\rho\le12$.  Then
\[
d=48q+\left\lceil\frac{48\rho}{13}\right\rceil
\]
and
\[
\left\lceil\frac{31n}{8}\right\rceil
=
50q+
\left\lceil
\frac{3q+31\rho}{8}
\right\rceil.
\]
If $q=0$, then $\eta(n,13)=0$, and the inequality follows from
\[
\frac{31}{8}>\frac{48}{13}.
\]
If $q\ge1$, then $\eta(n,13)\le1$ and
\[
\left\lceil\frac{3q+31\rho}{8}\right\rceil
\ge
\left\lceil\frac{31\rho}{8}\right\rceil
\ge
\left\lceil\frac{48\rho}{13}\right\rceil.
\]
The extra term $2q$ therefore absorbs the possible one-unit correction.  This proves the
claim.
\end{proof}

\begin{corollary}[$K_4$-free multipartite threshold]\label{cor:all-k4}
For every $n\ge1$ and every $r\ge4$, the value $f(n,r,4)$ is given by the following formula.

\[
f(n,r,4)=
\begin{cases}
\displaystyle
3n-\left\lceil\frac{2n}{3}\right\rceil,
& r=4,\\[3mm]
\displaystyle
\frac{2r}{3}\,n,
& r\ge6\text{ and }r\equiv0\pmod3,\\[3mm]
\displaystyle
\left(r-\left\lceil\frac r3\right\rceil\right)n,
& r\equiv2\pmod3,\\[3mm]
\displaystyle
(r-1)n-
\left\lceil
\frac{(m-1)(r-1)n}{r}
\right\rceil
-\eta(n,r),
& r=3m-2\equiv1\pmod3,\ r\ge7.
\end{cases}
\]
Here, in the last line, $m\ge3$ and $\eta(n,r)$ is defined above.
\end{corollary}

\begin{proof}
The case $r=4$ is the already solved $r=t+1$ case~\cite{HS,ST}.  Equation~(1.1) of
Lo--Treglown--Zhao~\cite{LTZ} gives
\[
f(n,4,4)=3n-\left\lceil\frac{2n}{3}\right\rceil.
\]
For $r\ge5$, the classes $r\equiv0\pmod3$ and $r\equiv2\pmod3$ are the previously known
cases: the former follows from Tur\'an's theorem together with the corresponding extremal construction, as summarized in~\cite{LTZ}, while the latter is
Theorem~1.2 of Lo--Treglown--Zhao~\cite{LTZ}.

It remains to consider $r\equiv1\pmod3$, so write $r=3m-2$.  The case $m=3$, namely
$r=7$, was established in~\cite{Kasugai}.  The cases $m=4$ and $m=5$,
namely $r=10$ and $r=13$, follow from Proposition~\ref{prop:r10r13}.  Finally, for
$m\ge6$ one has $r\ge16=2(3\cdot3-1)$, so Corollary~3.3 of Lo--Treglown--Zhao~\cite{LTZ} gives
\[
f(n,r,4)=\delta(n,r,3)
\]
for every $n$.  Apply Corollary~\ref{cor:t3-delta}.
\end{proof}

\begin{remark}
Lo--Treglown--Zhao~\cite{LTZ} obtained the $r=10$ and $r=13$ cases only under the restrictions
$n\ge60$ and $n\ge22$, respectively, and in the congruence class
$r\equiv1\pmod3$ their stated formula retained an additive constant.  The evaluation
of $\delta(n,r,3)$ removes both restrictions and determines that integer correction.
Together with the classical $r=4$ result and the $r=7$ case~\cite{Kasugai}, no admissible case remains in the balanced
multipartite $K_4$-free minimum-degree problem.
\end{remark}

\section{Concluding remarks}

Theorem~\ref{thm:main} removes the integer-rounding gap between the general upper bound for
$\delta(n,r,t)$ and the known constructions throughout
\[
r=mt-a,\qquad
2\le a\le\min\{m,t-1\}.
\]
The minimum matrix deficit exceeds the rounded lower bound $d$ by at most one, and the boundary
case is governed by a divisibility obstruction. The structural reason is that equality in the
column-sum inequalities forces every mixed row to use exactly two colours and forces the
resulting support graph to be a spanning star.

For $t=3$, the chromatic formula has an additional consequence: after combining the
$r=10$ and $r=13$ cases above with the $r=7$ result in~\cite{Kasugai} and the classical
congruence classes, the balanced multipartite $K_4$-free minimum-degree problem is determined
for every admissible $r\ge4$ and every $n\ge1$.

A natural next problem is to determine whether analogous formulae hold in the remaining
parameter range $m<a<t$, where the known constructions have a different form, and to determine
the full set of pairs $(r,t)$ for which $f(n,r,t+1)=\delta(n,r,t)$.

\section*{Declaration of Generative AI Use}

ChatGPT (OpenAI) was used to assist with reviewing the proof and translating the manuscript from Japanese into English. All mathematical arguments incorporated into the final manuscript were independently verified by the author, who assumes full responsibility for the content.

\end{document}